\documentclass[10pt]{article}
\font\smallit=cmti10

\usepackage{amssymb,latexsym,amsmath,epsfig,amsthm}
\usepackage{mathtools}
\usepackage{pifont}

\newtheorem{thm}{\bfseries Theorem}
\newtheorem{lem}[thm]{\bfseries Lemma}        
\newtheorem{cor}[thm]{\bfseries Corollary}     
\newtheorem{defn}[thm]{\bfseries Definition}

\makeatletter

\renewcommand\section{\@startsection {section}{1}{\z@}
{-30pt \@plus -1ex \@minus -.2ex}
{2.3ex \@plus.2ex}
{\normalfont\normalsize\bfseries}}

\renewcommand\subsection{\@startsection{subsection}{2}{\z@}
{-3.25ex\@plus -1ex \@minus -.2ex}
{1.5ex \@plus .2ex}
{\normalfont\normalsize\bfseries}}

\renewcommand{\@seccntformat}[1]{\csname the#1\endcsname. }

\makeatother

\def\stirling#1#2{\genfrac{\{}{\}}{0pt}{}{#1}{#2}}

\def\set#1{\{#1\}}
\def\B{\mathbb B}
\def\N{\mathbb N}
\def\Z{\mathbb Z}
\def\t1{\text{\ding{172}}}
\def\ex{\smallskip\noindent{\bf Example.}\ }
\def\endex{\hfill$\diamondsuit$\smallskip}

\begin{document}

\begin{center}
\uppercase{\bf Combinatorics of poly-Bernoulli numbers}
\vskip 20pt
{\bf Be\'ata B\'enyi}
     \\
{\smallit J\'ozsef E\"otv\"os College,
          Bajcsy-Zsilinszky u. 14., Baja, Hungary 6500}\\
{\tt benyi.beata@ejf.hu}\\
\vskip 10pt
{\bf Peter Hajnal}
     \\
{\smallit University of Szeged, Bolyai Insttitute,
          Aradi V\'ertan\'uk tere 1., Szeged, Hungary 6720}\\
{\tt hajnal@math.u-szeged.hu}\\
\end{center}

\vskip 30pt

\centerline{\bf Abstract}

\noindent
The $\B_n^{(k)}$ poly-Bernoulli numbers --- a natural generalization of
classical Bernoulli numbers ($B_n=\B_n^{(1)}$) --- were introduced by
Kaneko in 1997.  When the parameter $k$ is negative then $\B_n^{(k)}$
is a nonnegative number. Brewbaker was the first to give combinatorial
interpretation of these numbers. He proved that $\B_n^{(-k)}$ counts
the so called lonesum $0\text{-}1$ matrices of size $n\times k$.
Several other interpretations were pointed out. We survey these and
give new ones.  Our new interpretation, for example, gives a
transparent, combinatorial explanation of Kaneko's recursive formula
for poly-Bernoulli numbers.


\section{Introduction}\label{intro}

In the 17$^{\text{th}}$ century Faulhaber
\cite{Faulhaber} listed
the formulas giving the sum of the $k^{\text{th}}$
powers of the first $n$ positive integers when $k\leq 17$.
These formulas are always polynomials.
Jacob Bernoulli \cite{Bernoulli} 
realized the scheme in the coefficients
of these polynomials. Describing the coefficients
he introduced a new sequence of rational numbers. Later Euler
\cite{Euler}
recognized the significance of this sequence (that was connected
his several celebrated results). He named
the elements of the sequence as
Bernoulli numbers. For example Bernoulli numbers appear in the closed
formula for $\zeta(2k)$ (determining $\zeta(2)$ is the famous Basel problem,
that was solved by Euler).

In 1997 Kaneko (\cite{Kaneko})
introduced the poly-Bernoulli numbers.
Later Arakawa and Kaneko (\cite{KanekoRec}, \cite{AKMultipleZeta}) 
established strong connections
to multiple zeta values
(or Euler-Zagier sums, that were introduced and successfully used in 
different branches of mathematics earlier).

\begin{defn}{\rm (\cite{Kaneko})}
The poly-Bernoulli numbers 
$\set{\B_n^{(k)}}_{n\in\N, k\in Z}$
are defined by the following
exponential generating function
\[
\sum_{n=0}^\infty \B_n^{(k)}\frac{x^n}{n!}=
\frac{\text{Li}_k(1-e^{-x})}{1-e^{-x}},\qquad\text{for all }k\in\Z
\]
where
\[
Li_k(x)=\sum_{i=1}^\infty \frac{x^i}{i^k}.
\]
\end{defn}

The sequence $\set{\B_n^{(1)}}_n$ is just the 
Bernoulli numbers (with $\B_1^{(1)}=+\frac{1}{2}$).
Later Kaneko gave a recursive definition of the poly-Bernoulli numbers:

\begin{thm}{\rm (\cite{KanekoRec}, quoted in \cite{HaMa})}
\[
\B_n^{(k)}
=\frac{1}{n+1}\left(
\B_n^{(k-1)}-\sum_{m=1}^{n-1}\binom{n}{m-1}\B_m^{(k)}
\right),
\]
or equivalently
\[
\B_n^{(k-1)}=\B_n^{(k)}+\sum_{m=1}^n \binom{n}{m} \B_{n-(m-1)}^{(k)}.
\]
\end{thm}

We need a more combinatorial description of poly-Bernoulli numbers, 
given by Arakawa and Kaneko.

\begin{defn}
Let $a$ and $b$ be two natural numbers.
The Stirling number of the second kind
is the number of partitions of $[a]:=\set{1,2,\ldots,a}$
(or any set of $a$ elements)
into $b$ classes, and is denoted
by $\stirling{a}{b}$.
\end{defn}

Partitions can be described as equivalence relations (\cite{GKP}).

\begin{thm}{\rm (\cite{ArakawaKaneko})}
For any natural numbers $n$ and $k$ the following
formula holds
\[
\B_n^{(-k)}=\sum_{m=0}m!\stirling{n+1}{m+1}m!\stirling{k+1}{m+1}.
\]
\end{thm}

This theorem exhibits the fact that $\B_n^{(-k)}$
numbers are natural numbers. This formula
has initiated the combinatorial investigations
of poly-Bernoulli numbers.
There are several combinatorially 
described sequence of sets, such that their
size is $\B_n^{(-k)}$ (we call them {\it poly-Bernoulli families}). 
We can consider these statements
as alternative definitions of poly-Bernoulli numbers
or as answers to enumeration problems.

These combinatorial definitions give us the possibility
to explain previous identities
--- originally proven by algebraic methods ---
combinatorially.

The importance of the notion of poly-Bernoulli numbers is underlined by
the fact that there are several drastically
different combinatorial descriptions.

After reviewing the previous works 
we give a new 
poly-Bernoulli family. 
In our family we consider
$0\text{-}1$ matrices with certain
forbidden submatrix, what we call $\Gamma$. 
Our main result is that the number of
$n\times k$ 0-1 matrices that avoid the two submatrices 
$\begin{pmatrix} 1 & 1 \\ 1 & 1 \end{pmatrix}$
and
$\begin{pmatrix} 1 & 1 \\ 1 & 0 \end{pmatrix}$
is $\B_n^{(k)}$.
We present
a bijective
(\cite{Stanley}, \cite{StanleyBij})
proof of this.

There are several papers investigating
enumeration problems of matrices with some
forbidden substructure/pattern (see \cite{JuSe}, \cite{KMV}).
Their initial setups are different from ours.
In \cite{JuSe}
the order of rows and order of columns of matrices 
do not count. By permuting the rows and columns
of a matrix $M$
we obtain a matrix that includes the same patterns as $M$.
In \cite{KMV} the elements of the 
matrices do not play crucial role.
For example exchanging the 0's and 1's in a 0-1
matrix provides a matrix with the same 
patterns as the initial matrix.
Also their techniques are mainly analytical, that is 
different from our approach.

Finally some classical
results are explained combinatorially.
Among others a recursion for
$\B_{n}^{(k)}$ is proven based on our interpretation.
The proof is simple and straightforward.
The recursion is not new
--- as we mentioned earlier ---
(\cite{KanekoRec}), quoted in {\rm\cite{HaMa}).
All previous explanations of this recursive formula were
analytical. It was checked by algebraic manipulation
of generating functions.
This underlines a point: our main result can be
proven using a shortcut.
The number of $\Gamma$-free matrices satisfy
the mentioned recursion (see the end of our paper).
We know the same is true for poly-Bernoulli numbers
(\cite{KanekoRec}}), hence by induction
the main theorem follows.

From a combinatorialist's point of view
a simple checking the validity of a formula is not
satisfactory.
A combinatorial, bijective proof gives a deeper understanding of
the meaning of a formula.
Stanley (\cite{StanleyBij}) collected
several theorems, identities where
no bijective proof is known and urges
the combinatorialists to provide
one. Our paper is written in the spirit of Stanley's ideas.

The most recent research on poly-Bernoulli numbers
is mostly on extensions of the definition of poly-Bernoully
numbers and the number theoretical, analytical 
investigations of these extensions by analytical methods 
(\cite{BayadHamahata}, \cite{Kamano}, \cite{Sasaki}).
Our combinatorial approach is different from the
methods of these papers, and it might shed light
on some connections and might lead to new directions.


\section{Previous poly-Bernoulli families}

\subsection{The obvious interpretation}

Seeing the formula of Arakawa and Kaneko
one can easily come up with a combinatorial
problem such that the answer to it is
$\B_n^{(-k)}$.

Let $N$ be a set of $n$ elements and $K$ a set of $k$ elements.
One can think as $N=\set{1,2,\ldots,n}=:[n]$ and $K=[k]$.
Extend both sets with a special element:
$\widehat N=N\mathbin{\dot\cup}\set{n+1}$ and
$\widehat K=K\mathbin{\dot\cup}\set{k+1}$.
Take $\mathcal P_{\widehat N}$ a partition of $\widehat N$
and
 $\mathcal P_{\widehat K}$ a partition of $\widehat K$
with the same number of classes as $\mathcal P_{\widehat N}$.
Both partitions have a special class:
the class of the special element.
We call the other classes as {\sl ordinary} classes.
Let $m$ denote the number of ordinary classes
in $\mathcal P_{\widehat N}$ 
(that is the same as the number of ordinary classes
in $\mathcal P_{\widehat K}$).
Obviously $m\in\set{0,1,2,\ldots,\min\set{n,k}}$.
Order the ordinary classes arbitrarily
in both partitions. 
How many ways can we do this?

For fixed $m$ choosing $\mathcal P_{\widehat N}$
and ordering its ordinary classes can be done $m!\stirling{n+1}{m+1}$ ways.
Choosing the pair of ordered partitions can be done
$m!\stirling{n+1}{m+1}m!\stirling{k+1}{m+1}$ ways.
The answer to our question is
\[
\sum_{m=0}m!\stirling{n+1}{m+1}m!\stirling{k+1}{m+1}
=\B_n^{(-k)}.
\]

\subsection{Lonesum matrices}

\begin{defn}{\rm (\cite{Ryser})}
A $0\text{-}1$ matrix is lonesum iff
it can be reconstructed from its row and column sums.
\end{defn}

Obviously a lonesum matrix cannot contain the
\[
\begin{pmatrix}
1 & 0 \\
0 & 1 \\
\end{pmatrix},
\qquad
\begin{pmatrix}
0 & 1 \\
1 & 0 \\
\end{pmatrix}
\]
submatrices (a submatrix is a matrix that can be obtained 
by deletion of rows and
columns).  Indeed, in the case of the existence of one of the forbidden
submatrices we can switch it to the other one.  This way we obtain a
different matrix with the same row and column sums.  It turns out that
this property is a characterization \cite{Ryser}.

It is obvious that in a lonesum matrix
for two rows
`having the same row sum' is the same relation as `being equal'.
Even more, for rows $r_1$ and $r_2$
`the row sum in $r_1$ is at least the row sum in $r_2$' is the same as
`$r_1$ has $1$'s in the positions of the $1$'s of $r_2$'.
The same is true for columns.
Another easy observation is that changing the order of the rows/columns
does not affects the lonesum property.
These two observations guarantee that a lonesum matrix
can be rearranged by row/column order changes into a matrix where
the $1$'s in each row occupy a leading block of positions 
and these blocks are non-increasing as
we follow the natural order of rows. 
I.e.~substituting the $1$s with solid squares
we obtain a rotated stairs-like picture
(sequence of rectangles of height $1$, 
starting at the same vertical line
and with non-increasing width).
This type of diagrams are called
Young diagrams (see \cite{Stanley}
for further information).
This is also a characterization
of lonesum matrices. 

Look at the $1$s in a lonesum matrix
in the previous normal form, and consider it as a Young diagram:
the number of steps (block of rows with the same width)
is the number of different non-$0$
row sums and at the same time it is the number of different
non-$0$ column sums.
I.e.~the number of different non-$0$
row sums is the same as the number of different
non-$0$ column sums.

Let $M$ be a $0\text{-}1$ lonesum matrix
of size $n\times k$. Add a special
row and column with all $0$'s.
Let $\widehat M$ be the extended $(n+1)\times(k+1)$
matrix. `Having the same row sum' is an equivalence relation.
The corresponding partition has a special class, the set of
$0$ rows. By the extension we ensured that the
special class exists/non-empty. Let $m$ be the
number of non-special/ordinary classes.
The ordinary classes are ordered by their corresponding
row sums. In the same way we obtain an ordered 
partition of columns.
It is straightforward to prove that the two ordered
partitions give a coding of lonesum matrices.
This gives us the following theorem of Brewbaker, first 
presented in his MSc thesis.

\begin{thm}{\rm (\cite{BrewbakerThesis},\cite{Brewbaker})}
Let $\mathcal L_n^{(k)}$ 
denote the set of lonesum $0\text{-}1$
matrices of size $n\times k$. Then
\[
|\mathcal L_n^{(k)}|=
\sum_{m=0}m!\stirling{n+1}{m+1}m!\stirling{k+1}{m+1}
=\B_n^{(-k)}.
\]
\end{thm}

\subsection{Callan permutations and max-ascending permutations}

Callan \cite{Callan} considered the set $[n+k]$.
We call the elements $1,2,\ldots, n$ left-value
elements ($n$ of them) and $n+1,n+2,\ldots,n+k$ right-value elements
($k$ of them).
We extend our universe
with $0$, a special left-value element and
with $n+k+1$, a special right-value element.
Let $N=[n]$, $K=\set{n+1,n+2,\ldots,n+k}$,
$\widehat N=\set{0}\mathbin{\dot\cup} [n]$, $\widehat K=K\mathbin{\dot\cup}
\set{n+k+1}$,
Consider
\[
\pi:0,\pi_1,\pi_2,\ldots,\pi_{n+k},n+k+1
\]
a permutation
of $\widehat N\mathbin{\dot\cup}\widehat K$ with the restriction
that its first element is $0$ and its last element is $n+k+1$.
Consider the following equivalence relation/partition of left-values:
two left-values are equivalent iff
`each element in the permutation between them is a left-value'.
Similarly one can define an equivalence relation on the
right-values: `each element in the permutation between them is a right-value'.
The equivalence classes are just the ``blocks'' of left- and right-values 
in permutation $\pi$. The left-right reading of $\pi$ gives an ordering 
of left-value and right-value blocks/classes.
The order starts with a left-value block 
(the equivalence class of $0$, the special class)
and ends with a right-value block
(the equivalence class of $n+k+1$, the special class).
Let $m$ be the common number of ordinary left-value blocks and
ordinary right-value blocks.

Callan considered permutations 
such that in each block the numbers
are in increasing order.
Let $\mathcal C_n^{(k)}$
denote the set of these permutations.
For example
\[
\mathcal C_2^{(2)}=\{
012{\bf 3}{\bf 4}{\bf 5},
01{\bf 3}2{\bf 4}{\bf 5},
01{\bf 4}2{\bf 3}{\bf 5},
01{\bf 3}{\bf 4}2{\bf 5},
02{\bf 3}1{\bf 4}{\bf 5},
02{\bf 4}1{\bf 3}{\bf 5},
02{\bf 3}{\bf 4}1{\bf 5},
0{\bf 3}12{\bf 4}{\bf 5},
\]
\[
0{\bf 3}1{\bf 4}2{\bf 5},
0{\bf 3}2{\bf 4}1{\bf 5},
0{\bf 3}{\bf 4}12{\bf 5},
0{\bf 4}12{\bf 3}{\bf 5},
0{\bf 4}1{\bf 3}2{\bf 5},
0{\bf 4}2{\bf 3}1{\bf 5}
\}
\]
(the right-value numbers are in boldface).

It is easy to see that describing
a Callan permutation we need to give the two ordered partitions of 
the left-value and right-value elements.
Indeed, inside the blocks the `increasing' condition
defines the order, and the ordering of the classes
let us know how to merge the
left-value and right-value blocks.
We obtain the following theorem.

\begin{thm}{\rm (\cite{Callan})}
\[
|\mathcal C_n^{(k)}|=
\sum_{m=0}m!\stirling{n+1}{m+1}m!\stirling{k+1}{m+1}
=\B_n^{(-k)}.
\]
\end{thm}

\[
\star
\]

He, Munro and Rao \cite{HeMuRa} introduced the notion
of max-ascending permutations. 
This is (in some sense) a ``dual'' of the notion of Callan permutation.
We mention that \cite{OEIS} does not contain this description
of poly-Bernoulli numbers.
Now we give a slightly 
different version of max-ascending permutations
than the one presented in \cite{HeMuRa}.

Again we consider
\[
\pi:0,\pi_1,\pi_2,\ldots,\pi_{n+k},n+k+1
\]
permutations
of $\widehat N\mathbin{\dot\cup}\widehat K$ with the restriction
that its first element is $0$ and its last element is $n+k+1$.
We call the first $n+1$ elements of the permutation left-position 
elements ($0$ will be referred to as special left-position element).
Consider the following equivalence relation/partition of left-positions:
two left-positions, say $i$ and $j$, are equivalent iff
`any integer $v$ between $\pi(i)$ and $\pi(j)$ occupies 
a left position in $\pi$'.
Similarly one can define an equivalence relation on the
right-positions: 
`each value between the ones, that occupy the positions, 
is in a right-position'.
The max-ascending
property of a permutation is that in a class of positions
our numbers must be in increasing order.

For example consider the case when $n=4$ and $k=2$.
Let $\pi$ be the permutation $621534$.
We extend it with a first $0$ and last $7$:
\[
\bordermatrix{~   & 1 & 2 & 3 & 4 & 5 & 6 & 7 & 8 \cr
                  \pi: & 0 & 6 & 2 & 1 & 5 & {\bf 3} & {\bf 4} & {\bf 7} \cr}.
\]
The top row contains the positions, and the lower row shows the values
that are permuted. The numbers in bold face 
are the values in right positions ($k+1$ of them).
The two left positions $1$ and $3$ are equivalent, since
the corresponding values at these positions are $0$ and $2$;
only the value $1$ is between then, and that is at a left position.
The two left positions $1$ and $2$ are not equivalent, since
the corresponding values at these positions are $0$ and $6$;
the value $3$ is between them, but it is in a right position 
(in the $6^{\text{th}}$ place).
The equivalence classes
of left positions are $\set{1,3,4}$ and $\set{2,5}$.
The corresponding values standing in the first 
equivalence class
are $0,1,2$, and $5,6$ are the values that occupy
the positions of the second class.
The equivalence classes
of right positions are $\set{6,7}$ and $\set{8}$.
The permutation is not max-ascending permutation:
since at the positions $1,3,4$ (they form an equivalence class)
the values $0,1,2$ are not in increasing order.
In the case of
$512634$ the equivalence classes are the same (we permuted 
the values within positions forming an equivalence class).
It is a max-ascending permutation.

We give another example
\[
\mathcal A_2^{(2)}=\{
012{\bf 345},
013{\bf 425},
013{\bf 245},
014{\bf 235},
031{\bf 425},
031{\bf 245},
041{\bf 235},
\]
\[
023{\bf 415},
023{\bf 145},
034{\bf 125},
024{\bf 135},
024{\bf 315},
042{\bf 135},
042{\bf 315}
\},
\]
where boldface denotes the numbers at right-positions.

The definitions of Callan and max-ascending permutations 
are very similar. By exchanging the roles of position/value
we transform one of them into the other.
Specially if we consider our permutations as a bijection from
$\{0,1,\ldots,n+k,n+k+1\}$ to itself then 
"invert permutation" is a bijection from
$\mathcal C_n^{(k)}$ to $\mathcal A_n^{(k)}$.

\begin{thm}
Let $\mathcal A_n^{(k)}$ 
denote the set of
max-ascending permutations 
of $\set{0,1,2,\ldots,n+k+1}$. Then
\[
|\mathcal A_n^{(k)}|=
\sum_{m=0}m!\stirling{n+1}{m+1}m!\stirling{k+1}{m+1}
=\B_n^{(-k)}.
\]
\end{thm}

\subsection{Vesztergombi permutations}

Vesztergombi \cite{Vesztergombi} investigated permutations of $[n+k]$
with the property that $-k< \pi(i)-i<n$.
She determined a formula for their number.
Lov\'asz (\cite{Lovasz} Exercise 4.36.)
gives a combinatorial proof of this result.
Launois working on quantum matrices
slightly modified Vesztergombi's
set and realized the connection to poly-Bernoulli numbers.
Many of these results are summarized
in the following theorem.

\begin{thm}
Let $\mathcal V_n^{(k)}$
denote the set of
permutations $\pi$ of $[n+k]$ such that
$-k\leq \pi(i)-i\leq n$ for all $i$ in $[n+k]$.
\[
|\mathcal V_n^{(k)}|=
\sum_{m=0}m!\stirling{n+1}{m+1}m!\stirling{k+1}{m+1}
=\B_n^{(-k)}.
\]
\end{thm}

\subsection{Acyclic orientations of $K_{n,k}$}

Very recently P.~Cameron, C.~Glass and R. Schumacher
gave a new combinatorial interpretation
of poly-Bernoulli numbers.
(Peter Cameron published this in a note at his blog
on 19th of January, 2014 \cite{Cameron}.)

Let $K_{n,k}$ be the complete bipartite graph
with color classes of size $n$ and $k$.
An orientation of a graph is acyclic iff it
does not contain a directed cycle.

\begin{thm}{\rm (\cite{Cameron})}
The number of acyclic orientations of $K_{n,k}$
is $\B_n^{(-k)}$.
\end{thm}

Let $\mathcal O_n^k$ be the set of
acyclic orientations of $K_{n,k}$.
A simple graph theoretical observation
gives us that in a complete bipartite graph acyclicity 
is equivalent to `not having oriented $C_4$'.
The theorem is immediate from a bijection between $\mathcal O_n^k$
and $\mathcal L_n^k$ (i.e. set of lonesum matrices).
The bijection is easy and natural.
Identify the two parts of nodes in $K_{n,k}$
with the rows and columns of a matrix size $n\times k$.
Any edge has two possible orientations, hence
we can code an actual oriented edge by a bit (0/1)
according to its direction between the two
color classes.
The oriented graph can be coded by a $0\text{-}1$
matrix of size $n\times k$. Forbidding
oriented $C_4$'s is equivalent
to forbidding two submatrices of size $2\times 2$.
These matrices are the same as the ones in 
Ryser's characterization of
lonesum matrices.
So the desired bijection is just 
the simple coding we have described.


\section{A new poly-Bernoulli family}

Let $M$ be a $0\text{-}1$ matrix.
We say that three $1$s in $M$ form a $\Gamma$ configuration
iff two of them are in the same row
(one, let us say $a$, precedes the other) and the third is under
$a$.
I.e. the three $1$s form the upper left, upper right and lower left
elements of a submatrix of size $2\times 2$.
So we do not have any
condition on the lower right element of the submatrix of size $2\times 2$,
containing the $\Gamma$.

We will consider matrices without
$\Gamma$ configuration. Let $\mathcal G_n^{(k)}$
denote
the set of all $0\text{-}1$ matrices of size $n\times k$ without $\Gamma$.

The following theorem is our main theorem.

\begin{thm}
\[
|\mathcal G_n^{(k)}|=\B_n^{(-k)}.
\]
\end{thm}

The rest of the section is devoted to the
combinatorial proof of this statement.

The obvious way to prove our claim is to give
a bijection to one of the previous sets, where
the size is known to be $\B_n^{(-k)}$.
The obvious candidate is $\mathcal L_n^{(k)}$.
$\Gamma$-free matrices were considered from 
the point of extremal combinatorics (see \cite{FurediHajnal}).
It is known that $\Gamma$-free matrices
of size $n\times k$ contain
at most $n+k-1$ many $1$s.
Among Brewbaker's lonesum matrices (in contrast) 
there are some with many $1$s
(for example the all-$1$ matrix) and there 
are others with few $1$s.
We do not know straight, simple bijection between lonesum matrices
and matrices with no $\Gamma$.
Instead, we follow the obvious scheme:
we code $\Gamma$-free matrices with two partitions and two orders.
From this and from the previous bijections one can construct
a direct bijection between the two sets of matrices 
but that is not appealing.

\noindent{\it Proof.\ }
Let $M$ be a $0\text{-}1$ matrix of size $n\times k$.  We say that a
position/element has {\sl height} $n-i$ iff it is in the $i^{\text{th}}$
row.  The {\sl top-$1$} of a column is its $1$ element of maximal height.
The {\sl height of a column} is the height of 
its top-$1$ or $0$, whenever it
is a $0$ column.

Let $M$ be a matrix without $\Gamma$ configuration.  Let $\widehat M$
be the extension of it with an all $0$s column and row.
(We have defined the {\sl height of all-$0$ columns} to be
$0$. In $\widehat M$ non-$0$ columns have $0$ at the bottom, hence 
their heights are at least $1$.)  `Having the
same height' is an equivalence relation on the set of columns in
$\widehat M$.  The class of the additional column is the set of $0$
columns (that is not empty since we work with the extended matrix).
We call the class of the additional 
column `the special class'. Its elements
are the special columns. So special column means `all-$0$ column'.
The additional column in $\widehat M$ ensures that we
have this special class. The other classes are the ordinary classes.
Let $m$ be the number of the ordinary  classes. These $m$ classes
partition the set of non-$0$ columns. The total number of equivalence classes
is $m+1$.

In order to clarify the details after the formal description
we explain the steps on a specific example.
$\diamondsuit$ denote the end of example, when we 
return to the abstract discussion,

\ex
$M$ is a $\Gamma$-free matrix:
\[
M=\begin{pmatrix}
0&1&0&0&0&0&0&0\\
0&0&0&0&0&0&0&1\\
1&0&1&0&0&0&0&0\\
0&1&0&1&0&1&0&0\\
0&0&1&0&0&0&0&0\\
0&0&0&0&0&0&1&0
\end{pmatrix},
\widehat M=\left(\begin{array}{cccccccc|c}
0&1&0&0&0&0&0&0&0\\
0&0&0&0&0&0&0&1&0\\
1&0&1&0&0&0&0&0&0\\
0&1&0&1&0&1&0&0&0\\
0&0&1&0&0&0&0&0&0\\
0&0&0&0&0&0&1&0&0\\
\hline
0&0&0&0&0&0&0&0&0\\
\end{array}\right)
\]
$\widehat M$ is ``basically the same''
as $M$. It only contains an additional all-$0$ column (the last one),
and an additional all-$0$ row (the last one).
In $\widehat M$ each column has a height. The height depends on the 
position of the top-$1$ of the considered column
(it counts how many positions are under it).
The all-$0$ 
column has height $0$. We circled the top-$1$s
and marked the height of the columns at the upper 
border of our matrix $\widehat M$:
\[
\widehat M=\bordermatrix{
~&4  &  6&4  &3  &0&3  &1  &5  &0\cr
~&0  &\t1&0  &0  &0&0  &0  &0  &0\cr
~&0  &0  &0  &0  &0&0  &0  &\t1&0\cr
~&\t1&0  &\t1&0  &0&0  &0  &0  &0\cr
~&0  &1  &0  &\t1&0&\t1&0  &0  &0\cr
~&0  &0  &1  &0  &0&0  &0  &0  &0\cr
~&0  &0  &0  &0  &0&0  &\t1&0  &0\cr
~&0  &0  &0  &0  &0&0  &0  &0  &0\cr
},
\]
`Having the same height' is an equivalence relation among the columns. 
In our example there are  six different heights considering all
the columns: $0,1,3,4,5,6$.
Two columns have height $0$ (they are the two all-$0$ columns).
One of them is the additional column of $\widehat M$,
the other was present in $M$. They
are the special columns, forming the special
class of our equivalence relation on columns. The other five heights
define
five ordinary classes. One of these classes is formed by the first and third
column, they are the columns with height $4$.
\endex

Take $\mathcal C$, any non-special class 
of columns (the columns in $\mathcal C$ 
are ordered as the indices order
the whole set of columns). Since our matrix does not contain $\Gamma$
all columns but the last one has only one $1$ (that is necessarily the top-$1$)
of the same height. We say that the last elements/columns of non-special classes are
{\sl important columns}. Important columns in $\widehat M$ form a
submatrix $M_0$ of size $(n+1)\times m$.

In $M_0$ the top-$1$s are called {\sl important elements}.
In each row without top-$1$ the {\sl leading $1$} 
(the $1$ with minimal column index)
is also called {\sl important $1$}.
So in all non-$0$ rows of $M_0$ there is
exactly one important $1$. 

Each row has an `indentation': the position of the
important $1$, i.e.~last top-$1$ if the
row contains a top-$1$, otherwise the position of the first $1$ (or $0$ if
the row is all $0$s). The row indentations determine a partition of
the set of rows.

The two partitions have the same number of parts, namely, $m+1$ where $m$
was introduced when describing the column partition.

The last top-$1$s are in different rows and columns, hence determine an
$m\times m$ submatrix which becomes a permutation matrix if all entries
except the last top-$1$s are zeroed out. This permutation matrix
determines an identification of the ordinary 
column classes and ordinary row classes.

\ex
$M_0$ contains the last columns of the ordinary equivalence
classes. In our example it has $5$ columns
(the upper border of our example we
see the common height of the column class, and the original index of each row).

The top-$1$s are circled. There are two rows without top-$1$. 
The last row is all-$0$,
the other in not all-$0$. Its leading $1$ is in bold face.
We also marked
(at the left border of $M_0$) the indentations
of rows. The indentation/label of an all-$0$ row is $0$:
\[
M_0=\bordermatrix{
~&{6\atop 2^\text{nd}}&{4\atop 3^\text{rd}}&{3\atop 6^\text{th}}&{1\atop 7^\text{th}}&{5\atop 8^\text{th}}\cr
1&\t1&0  &0  &0  &0  \cr
5&0  &0  &0  &0  &\t1\cr
2&0  &\t1&0  &0  &0  \cr
3&1  &0  &\t1&0  &0  \cr
2&0  &\text{\bf 1}  &0  &0  &0  \cr
4&0  &0  &0  &\t1&0  \cr
0&0  &0  &0  &0  &0  \cr
}.
\]
Note that
since each column has a top-$1$ we have $1,2,3,4,5$ as labels
(all-$0$ columns form the special column class that 
has no representative in $M_0$).
The last row of
$M_0$ is an all-$0$ row, hence we have the $0$ label too.

The top-$1$s define an identification
of ordinary row and column classes.
We use the letters $a,b,c,d,e$
for the identified
row/column classes.
$s$ marks the special rows (in our case there
is only one special row, the last row):
\[
M_0=\bordermatrix{
~&{a\atop 2^\text{nd}}&{b\atop 3^\text{rd}}&{c\atop 6^\text{th}}&{d\atop 7^\text{th}}&{e\atop 8^\text{th}}\cr
a&\t1&0  &0  &0  &0  \cr
e&0  &0  &0  &0  &\t1\cr
b&0  &\t1&0  &0  &0  \cr
c&1  &0  &\t1&0  &0  \cr
b&0  &\text{\bf 1}  &0  &0  &0  \cr
d&0  &0  &0  &\t1&0  \cr
s&0  &0  &0  &0  &0  \cr
}.
\]
\endex

A partition of columns into $m+1$ classes,
and partition of rows into $m+1$ classes, and
a bijection between the non-special row- and column-classes
--- after fixing $m$ ---
leaves
\[
m!\stirling{n+1}{m+1}\stirling{k+1}{m+1}
\]
possibilities.
This information (knowing the two partitions and
the correspondence) codes a big part of matrix $\widehat M$:

We know that the columns and rows of the special classes
are all $0$s. A non-special column class $\mathcal C$ has a corresponding
class of rows. The top row of the corresponding row class
gives us the common height of the columns in $\mathcal C$.
So we know each non-important columns (they have only one $1$, 
defining its known height).
We narrowed the unknown $1$s of $M$ into the non-$0$ rows
of $M_0$. Easy to check that $\widehat M$ contains $\Gamma$ iff $M_0$
contains one.

\ex
Let us consider our example. $m=5$ and the row/column partitions are marked
at the left and upper border of our matrix. We use $a,b,c,d,e,s$ as
names for the classes in both cases (hence the classes of the two partitions
are identified). $s$ denotes the two special classes (the class of 
the last row and last column). 
\[
\widehat M=\bordermatrix{
~&b&a&b&c&s&c&d&e&s\cr
a&~&~&~&~&~&~&~&~&~\cr
e&~&~&~&~&~&~&~&~&~\cr
b&~&~&~&~&~&~&~&~&~\cr
c&~&~&~&~&~&~&~&~&~\cr
b&~&~&~&~&~&~&~&~&~\cr
d&~&~&~&~&~&~&~&~&~\cr
s&~&~&~&~&~&~&~&~&~\cr
}
\]
We can recover a large portion of $\widehat M$ from this information.
The class of the last column, and the class of the
last row are the two special classes (named by $s$).
We must have all $0$s
in these rows/columns.

The columns of $M_0$ are the last columns of
the ordinary classes.
For example the top-$1$s in the column class $a$ 
can be decoded from the rows belonging to class $a$:
The highest row with label $a$ marks the
common row of top-$1$s in columns labelled by $a$.
This way we can  determine the common heights of the
ordinary column classes, hence
recover the top-$1$s.
Then we know all elements above a top-$1$
must have value $0$.
All columns with label $a$ but the last one contains 
only its top-$1$ as non-$0$ element.

In our example we sum up the information gained so far (the top border
contains the recovered heights and
labels for the columns from $M_0$):
\[
\widehat M=\bordermatrix{
~&4&{6\atop M_0}&{4\atop M_0}&3&0&{3\atop M_0}&{1\atop M_0}&{5\atop M_0}&0\cr
b&0&1&0&0&0&0&0&0&0\cr
f&0&~&0&0&0&0&0&1&0\cr
a&1&~&1&0&0&0&0&~&0\cr
c&0&~&~&1&0&1&0&~&0\cr
a&0&~&~&0&0&~&0&~&0\cr
e&0&~&~&0&0&~&1&~&0\cr
d&0&0&0&0&0&0&0&0&0\cr
}
\]
Knowing the top-$1$s enable us to recover the indentations
(relative to $M_0$)
belonging to the ordinary row classes.
If there is a row without a top-$1$, then from its row
label we know the position of its first $1$ 
(hence we know that in that row at previous positions we have $0$s):
\[
\widehat M=\bordermatrix{
~&4&{6\atop M_0}&{4\atop M_0}&3&0&{3\atop M_0}&{1\atop M_0}&{5\atop M_0}&0\cr
1/M_0&0&1&0&0&0&0&0&0&0\cr
5/M_0&0&~&0&0&0&0&0&1&0\cr
2/M_0&1&~&1&0&0&0&0&~&0\cr
3/M_0&0&~&~&1&0&1&0&~&0\cr
2/M_0&0&0&1&0&0&~&0&~&0\cr
4/M_0&0&~&~&0&0&~&1&~&0\cr
0/M_0&0&0&0&0&0&0&0&0&0\cr
}
\]
\endex

Note that there are many positions where we do not
know
the elements
of our matrix (all are located in $M_0$). Also when counting
the possibilities we have a missing $m!$
factor. The rest of the proof shows
that filling in the missing elements
(resulting a $\Gamma$-free matrix) can be done $m!$ many ways.

Now on we concentrate on $M_0$ (that is where the unknown
elements are).
The positions of important ones are known.
In each column of $M_0$ there is a lowest important $1$.
We call them {\sl crucial} $1$s. (Specially crucial $1$s
are important $1$s too.)
We have $m$ many crucial $1$s, one is in each column of $M_0$.
A $1$ in $M_0$ that is non-important is called {\sl hiding}
$1$.

\begin{lem}
Consider a hiding $1$ in $M_0$. Then exactly one of the following two
possibilities holds:
\begin{itemize}
\item[(1)]
there is a crucial $1$ above it and a top-$1$ to the right of it,
\item[(2)]
there is a crucial $1$ on its left side (and of course a top-$1$ above it).
\end{itemize}
\end{lem}

\ex
The following figure exhibits
the two options: $1^{(h1)}$, $1^{(h2)}$ are two hiding
$1$s, corresponding case (1) and case (2) respectively. 
Top-$1$s are the circled $1$s.
The two hiding $1$s "share" the crucial $1$
in the lemma.
\[
\begin{pmatrix}
       & \vdots  &                &            &   &  \\
       &         & \ldots               & \t1     &   &     \\
       & \vdots  &                & \uparrow &   &  \\
 \dots & 1^{(c)} &     \leftarrow &  1^{(h2)}  & &     \\
       & \uparrow  &                & \vdots  &         &     \\
       &1^{(h1)} &     \rightarrow & & \rightarrow & \t1 \\
       & \vdots  &                &         &   &     \\
\end{pmatrix}
\]
\endex

\begin{proof}
Let $h$ be a hiding $1$ in $M_0$. 

First, assume that the row of $h$
does not contain a top-$1$.
Then the first $1$ in this row ($f$)
is an important $1$ (hence it differs from $h$).
Since the matrix is $\Gamma$-free, we cannot have
a $1$ under $f$, i.e.~$f$ is a crucial $1$.
$h$ is not important, so it is not a top-$1$.
The top-$1$ in its column must be above it. We obtained that 
case (2) holds.

Second, assume that the row of $h$ contains a top-$1$, $t$.
If $t$ is on the left of $h$ then the forbidden $\Gamma$ ensures
that under $t$ there is no other $1$. Hence $t$ is crucial 
and case (2) holds again.
If $t$ is on the right of $h$ then the forbidden $\Gamma$ ensures
that under $h$ there is no other $1$. Hence the lowest important $1$ 
in the column of $h$ (a crucial $1$) is above of it. Case (1) holds.

(1) and (2) cases are exclusive since if both are satisfied
then $h$ has a crucial $1$ on its left and a top-$1$ on its right.
That is impossible since the $1$s in a row of a top-$1$ are not even
important.
\end{proof}

Let $h$ be a hiding $1$.
There must be a unique crucial $1$
corresponding to it:
If $h$ satisfies case $(1)$, then
it is the crucial one above it.
If $h$ satisfies case $(2)$, then
it is the crucial one on the left side of it.
In this case we say that this crucial $c$
is {\it responsible} for $h$.

Take a crucial $1$ in $M_0$, that we call
$c$. For any top-$1$, $t$ that comes in 
a later column and it is higher than $c$
the position in the row of $c$ under $t$ we call {\sl questionable}.
Also for any top-$1$, $t$ that comes in a later column and it is lower than $c$
the position in the column of $c$ before $t$ we call {\sl questionable}.
In $M_0$ there are $m$ many crucial $1$. If $c$ is in the $i^{\text{th}}$
column, then there are $m-i$ column that comes later 
and each defines one questionable
position.

The meaning of the lemma is that
each hiding $1$ must be in a questionable position.

\ex
We continue our previous example (but only $M_0$ is followed on). 
We
circled the top-$1$s in $M_0$.
We added an index $c$ to crucial $1$s.
(All the important $1$s are identified so we
are able to locate these elements.
\[
\widehat M=\begin{pmatrix}
\t1^{(c)}&0&0&0&0\\
~&0&0&0&\t1^{(c)}\\
~&\t1&0&0&~\\
~&~&\t1^{(c)}&0&~\\
0&1^{(c)}&~&0&~\\
~&~&~&\t1^{(c)}&~\\
0&0&0&0&0\\
\end{pmatrix}
\]

According to our argument,
for the crucial $1$ in the first column there are
four top-$1$s in the later columns and
there are four questionable positions
corresponding to them.
We mark them as $?_1$.
Similarly,
three questionable positions belongs 
to the crucial $1$ in the second column, marked as $?_2$.
(If there is  hiding $1$ in one of these positions then
the crucial $1$ of the second column 
would be responsible to it).
We put $?$ to each questionable position and
add an index marking the column of the crucial $1$ that
is connected to it:
\[
M_0=\begin{pmatrix}
\t1_c&0      &0        &0        &0        \\
?_1  &0      &0        &0        &\t1^{(c)}\\
?_1  &\t1    &0        &0        &~        \\
?_1  &~      &\t1^{(c)}&0        &?_3      \\
0    &1^{(c)}&?_2      &0        &?_2      \\
?_1  &?_2    &?_3      &\t1^{(c)}&?_4      \\
0&0&0&0&0\cr
\end{pmatrix}
\]
Note that there are positions that are not questionable.
The lemma says there can not be a hiding $1$. Indeed, a $1$
at these positions would create a $\Gamma$.
\endex

First rephrase our lemma:

\begin{cor}
All hiding $1$s are in questionable positions.
\end{cor}

It is obvious that we have $(m-1)+(m-2)+\ldots+2+1$
many questionable positions (to the crucial $1$ in the $i^\text{th}$
column there are $m-i$ many questionable position is assigned).

Easy to check that if we put the important $1$s into $M_0$ and
add a new $1$ into a questionable position then we
won't create a $\Gamma$ configuration.
The problem is that the different questionable positions
are not independent.

\begin{lem}
There are $m!$ ways to fill the questionable positions with $0$s
and $1$s without forming a $\Gamma$.
\end{lem}

\begin{proof}
Let $c$ be a crucial $1$. We divide the set of questionable
positions that corresponds to $c$, depending their positions
relative to $c$ into two parts: Let $R_c$ be the set of questionable
positions in the row of $c$, that is right from $c$.
Let $D_c$ be the set of questionable positions in the column of $c$, that is
down from $c$.

The following two observation is immediate:
\begin{itemize}
\item[(i)]
At most one of $R_c$ and $D_c$ contains a $1$.
\item[(ii)]
If $D_c$ contains a $1$ (hence $R_c$ is empty), then it contains
only one $1$.
\end{itemize}
Indeed, if the two claims are not true then 
we can easily recognize a $\Gamma$.

For each crucial $c$
describe the following `piece of information':
$I_1$: the position of the first $1$ in $R_c$ or 
$I_2$: the position of the only one $1$ in $D_c$ (this informs 
us that $R_c$ contains only $0$s) or
$I_3$: say ``all the positions of $R_c\cup D_c$ contain $0$''.

If $c$ comes from the first column of $M_0$, then
we have $m$ many outcomes for this piece of information.
$m-1$ many of these are such that one position of a $1$ is
revealed (the first $1$ in $R_c$ or the the only one $1$ in $D_c$). 
Placing a $1$ there 
doesn't harm the $\Gamma$-free
property of our matrix. One possible outcome of the information
is the one that reveals that there is
no $1$ in $R_c\cup D_c$.

\ex
In our example let $c$ be the crucial $1$ of the first column, i.e.~the
first element of the first row:
\[
M_0=\begin{pmatrix}
c=\t1^{(c)}&0      &0        &0        &0        \\
?_1        &0      &0        &0        &\t1^{(c)}\\
?_1        &\t1    &0        &0        &0        \\
?_1        &0      &\t1^{(c)}&0        &?_3      \\
0          &1^{(c)}&?_2      &0        &?_2      \\
?_1        &?_2    &?_3      &\t1^{(c)}&?_4      \\
0          &0      &0        &0        &0        \\
\end{pmatrix}
\]
$R_c$ is empty, $D_c$
contains four positions from the first column (2$^\text{nd}$,
3$^\text{rd}$, 4$^\text{th}$, 6$^\text{th}$).
So when we reveal the above mentioned information about $c$
then we have the following $5$ possibile outcomes:
\begin{itemize}
\item[$I_2$(2):]
``The only $1$ under $c$ is in the second row and 
there is no $1$ in the row of $c$.''
\item[$I_2$(3):]
``The only $1$ under $c$ is in the third row and 
there is no $1$ in the row of $c$.''
\item[$I_2$(4):]
``The only $1$ under $c$ is in the fourth row and 
there is no $1$ in the row of $c$.''
\item[$I_2$(6):]
``The only $1$ under $c$ is in the sixth row and 
there is no $1$ in the row of $c$.''
\item[$I_3$:]
``There is no $1$ under and after $c$.''
\end{itemize}

Let us assume that we get the the third possibility 
($I_2(4)$)
as the additional information. Then we can continue filling the missing elements
of $M_0$:
\[
M_0=\begin{pmatrix}
1&0&0&0&0\\
0&0&0&0&\t1^{(c)}\\
0&\t1&0&0&0\\
1&0&\t1^{(c)}&0&?_3\\
0&c=1^{(c)}&?_2&0&?_2\\
0&?_2&?_3&\t1^{(c)}&?_4\\
0&0&0&0&0\\
\end{pmatrix}
\]
Now let $c$ the crucial $1$ in the second column.
$R_c$ contains two positions
(3$^{\text{rd}}$ and 5$^{\text{th}}$ column), $D_c$
contains one position from the column of the actual $c$
(the one in the 6$^\text{th}$ row).
So when we reveal the above mentioned information about $c$
then we have the following $4$ possibilities:
\begin{itemize}
\item[$I_1(3)$:]
``The first $1$ after $c$ is in  the third column and there is
no $1$ in $D_c$.''
\item[$I_1(5)$:]
``The first $1$ after $c$ is in  the fifth column and there is
no $1$ in $D_c$.''
\item[$I_2(6)$:]
``The only $1$ under $c$ is in the sixth row and 
there is no $1$ in the row of $c$.''
\item[$I_3$:]
``There is no $1$ under and after $c$.''
\end{itemize}
Let us assume that we get the the first possibility 
($I_1(3)$)
as an additional information. The hidden $1$ that is 
revealed is under the crucial $1$ ($c'$) of its column.
The $\Gamma$-free property of $M$ (and hence $M_0$)
guarantees that cannot be a hiding $1$ after $c'$.
Again we summarize the information gained:
\[
M_0=\begin{pmatrix}
1&0&0&0&0\\
0&0&0&0&\t1^{(c)}\\
0&\t1&0&0&0\\
1&0&c'=\t1^{(c)}&0&0\\
0&1^{(c)}&1&0&?_2\\
0&0&?_3&\t1^{(c)}&?_4\\
0&0&0&0&0\\
\end{pmatrix}
\]
In the third column the old crucial $1$ ($c'$)
will be replaced by the $1$, ($c$) revealed by the 
previous information.
\[
M_0=\begin{pmatrix}
1&0&0&0&0\\
0&0&0&0&\t1^{(c)}\\
0&\t1&0&0&0\\
1&0&c'=\t1&0&0\\
0&1^{(c)}&c=1^{(c)}&0&?_3\\
0&0&?_3&\t1^{(c)}&?_4\\
0&0&0&0&0\\
\end{pmatrix}
\]
\endex

If we get $I_2$ or $I_3$ then we know all the elements at 
the questionable positions corresponding to $c$. In this 
case we can inductively
continue and finish the description of $M$.
If the information, we obtain is $I_1$ then our knowledge about the
$1$s at the questionable positions corresponding to $c$ is not complete.
But we can deduce many additional information.

Assume that $I_1$ says that on the right of $c$ the first
$1$ in questionable position is in the $j^{\text{th}}$ column.
Let $\widetilde c_j$ be the position
of this $1$. The position $c_j$ is above of it.
We know that $R_i$ doesn't contain a $1$ (indeed, that would form 
a $\Gamma$ with the $1$s at $c_j$ and at $\widetilde c_j$).
For similar reasons also we cannot have a $1$ 
at a questionable position between
$c_j$ and $\widetilde c_j$.

This knowledge guarantees that we can substitute $c_j$ with
$\widetilde c_j$ ($\widetilde c_j$ will be a crucial $1$ substituting $c_j$).
The corresponding questionable positions will be the
questionable positions that are down and right from
it. We still encounter all the hiding ones (there must be at
the questionable positions corresponding to
the crucial $1$s, we didn't confronted yet).
So we can induct.

The above argument
proves that
any element
of  $\{1,2,\ldots,m\}\times\{1,2,\ldots,m-1\}\times\{1,2\}\times\{1\}$
codes the
outcome of the information revealing
process, hence a $\Gamma$-free completion of our previous knowledge.
The $i^{\text{th}}$ component of the code says
that in the $i^{\text{th}}$ column
of $M_0$ which information on the actual crucial $1$
is true.
Our previous argument just describe how to do the first
few steps of the decoding 
and how to recursively continue it.
\end{proof}

The lemma finishes the enumeration
of $\Gamma$-free $0\text{-}1$ matrices of size $n\times k$.
Also finishes a description of a constructive bijection from
$\mathcal G_n^{(k)}$ to the obvious poly-Bernoulli set.
Our main theorem is proven a bijective way.


\section{Combinatorial proofs}

\begin{thm}
\[
\B_n^{(-k)}=\B_k^{(-n)}.
\]
\end{thm}

The relation originally was proven by Kaneko.
It is obvious from any of the combinatorial
definitions. Arakawa--Kaneko formula also exhibits this symmetry
an algebraic way.

\begin{thm}
\[
\B_n^{(-k)}=\B_n^{(-(k-1))}+
\sum_{j=1}^n {n\choose j}\B_{n-(j-1)}^{(-(k-1))}.
\]
\end{thm}

\begin{proof}
Our main theorem gives that
$\B_n^{(-k)}$ counts the $\Gamma$-free matrices of size $n\times k$.

Each row of a $\Gamma$-free matrix
\begin{itemize}
\item[A.] 
starts with a $0$, or
\item[B.] 
starts with a $1$, followed only by $0$s, or
\item[C.] 
starts with a $1$, and contains at least one more $1$.
\end{itemize}
Let $j$
denote the number of rows of type B/C.

If $j=0$, then the first column is all-$0$ column,
and it has $\B_n^{(-(k-1))}$
many extensions as $\Gamma$-free matrix.

If $j\geq 1$, then we must choose the 
$j$ many rows of type B/C.
Our decision describes the first column of
our matrix. The first $j-1$ many chosen rows 
cannot contain any other
$1$, since a $\Gamma$ would appear. I.e.~they are type B,
and completely described.

The further elements (a submatrix of size $(n-j+1)\times (k-1)$)
can be filled with an arbitrary $\Gamma$-free matrix.
The recursion is proven.
\end{proof}

We can state
the theorem (without a reference to the main theorem)
as a recursion for $|\mathcal G_n^{(k)}|$.
Since the same recursion
is known for
$\B_n^{(-k)}$, an easy induction proves the main theorem.
Our first proof, the main part of this paper
is purely combinatorial and explains a previously
known recursion without algebraic
manipulations of generating functions.

\begin{thm}
\[
\sum_{n,k\in\N: n+k=N} (-1)^n \B_n^{(-k)}=0.
\]
\end{thm}

\begin{proof}
We use Callan's description of poly-Bernoulli numbers.
We consider Callan permutations of $N$ objects 
(the extended base set has
size $N+2$).
We underline that to speak about Callan permutations
we must divide the $N$ objects into left and right value category.
For this we need to write $N$ as a term sum: $n+k$.

For technical reasons we change the base set
of our permutations.
The extended left values remain $0,1,2,\ldots,n$,
the extended right values will be
$\bf{1},\bf{2},\ldots,\bf{k},\bf{k+1}$.
We note that
the calligraphic distinction between the left and right values allows
us to use any $n+1$ numbers for the extended
left values, and the same is true for the right values.

The combinatorial content of the claim is that if we consider
Callan permutation of $N$ objects (with all possible $n+k$ partitions),
then those where the number of left values is even has the
same cardinality as 
those where the number of left values is odd.

$\B_n^{(-k)}$ is the size of $\mathcal C_n^{(k)}$.
We divide it into two subsets according to the type
of the element following the leading $0$.
Let
$\mathcal C_n^{(k)}(\ell)$ be the set of those
elements from $\mathcal C_n^{(k)}$, where the leading $0$ is followed
by a left value element.
Let
$\mathcal C_n^{(k)}(r)$ be the set of those
elements from $\mathcal C_n^{(k)}$, where the leading $0$ is followed
by a right value element.

We give two examples:
\[
\mathcal C_3^{(1)}(\ell)=\{
01{\bf 1}23{\bf 2},
02{\bf 1}13{\bf 2},
03{\bf 1}12{\bf 2},
012{\bf 1}3{\bf 2},
013{\bf 1}2{\bf 2},
023{\bf 1}1{\bf 2},
0123{\bf 12}
\},\]
\[
\mathcal C_2^{(2)}(r)=\{
0{\bf 1}1{\bf 2}2{\bf 3},
0{\bf 1}2{\bf 2}1{\bf 3},
0{\bf 1}12{\bf 23}
0{\bf 12}12{\bf 3},
0{\bf 2}1{\bf 1}2{\bf 3},
0{\bf 2}2{\bf 1}1{\bf 3},
0{\bf 2}12{\bf 13},
\}.\]

We will describe a 
$\varphi: \mathcal C_n^{(k)}(\ell)\to \mathcal C_{n-1}^{(k+1)}(r)$
bijection.
(Hence we will have a
$\varphi: \mathcal C_n^k(r)\to \mathcal C_{n+1}^{(k-1)}(\ell)$
bijection too.)
Our map reverses the parity of the number of left values and completes
the proof.

The bijection goes as follows:
Take a permutation from
$\mathcal C_n^{(k)}(\ell)$.
Find `1' in the permutation.
It follows the leading $0$ or
it will be the first element of a block of left values
(that is preceded by a block of right value, say $R$).
In the first case we substitute $1$ by ${\bf 0}$.
In the second case
we also substitute $1$ by ${\bf 0}$ but additionally
we move the $R$ block
right after the leading $0$.

We warn the reader that the image permutations
have extended left values $0,2,\ldots,n$ and extended
right values $\bf{0,1,\ldots,k+1}$. This change does not
effect the essence.

An example helps to digest the technicalities:
\[
012{\bf 1}3{\bf 2} \to 0{\bf 0}2{\bf 1}3{\bf 2} \equiv
0{\bf 1}1{\bf 2}2{\bf 3},
\]
\[
03{\bf 1}12{\bf 2} \to 0{\bf 1}3{\bf 0}2{\bf 2} \equiv
0{\bf 2}2{\bf 1}1{\bf 3}.
\]

The inverse of our map can be easily constructed.
It must be based on ${\bf 1}$. The details are left to the reader.
\end{proof}

\section{Conclusion}

We presented a summary of previous descriptions of
poly-Bernoulli numbers, including a new one. 
Our list isn't as respectful
as Stanley's list for Catalan numbers, but suggests
that poly-Bernoulli numbers are natural and central.
The last two proofs are the only combinatorial explanations
(as far we know) for basic relationships for poly-Bernoulli numbers.
We expect further combinatorial
definitions and proofs, enriching the understanding
poly-Bernoulli numbers.

\end{document}